\documentclass[a4paper]{amsart}
\usepackage{amsmath,amssymb,amsthm}
\usepackage{graphicx} 
\usepackage{color}

\renewcommand{\div}{{\rm div}\,}

\def\be{\mbox{\boldmath $e$}}
\def\br{\mbox{\boldmath $R$}}

\def\etab{\mbox{\boldmath $\eta$}}

\def\bN{\mbox{\boldmath $N$}}
\def\R{\mathbb{R}}
\def\bR{\mbox{\boldmath $r$}}
\def\bt{\mbox{\boldmath $t$}}

\def\bV{\mbox{\boldmath $V$}}

\def\1{\mbox{\boldmath $1$}}
\def\0{\mbox{\boldmath $0$}}
\def\cH{\mathcal{H}}
\def\cK{\mathcal{K}}
\def\<{\langle}
\def\>{\rangle}

\theoremstyle{plain}
\newtheorem{thm}{Theorem}[section]

\newtheorem{lem}[thm]{Lemma}

\newtheorem{rmk}[thm]{Remark}

\title[First variation]{A derivation of first variation formulas from the strain-displacement relations in thin shell theory}
\author{Yoshiki Jikumaru}
\address{Faculty of Information Networking for Innovation And Design, Toyo University, 1-7-11 Akabanedai, Kita-ku, Tokyo, 115-8650, Japan}
\email{jikumaru@toyo.jp}

\begin{document}
\maketitle
\begin{abstract}
In this paper, we derive the first variation formulas for surfaces in 3-dimensional Euclidean space by using the ``strain-displacement relations'' known in thin shell theory.
For applications to architectural surface design, we focus on the objective function which has linear Weingarten surfaces as stationary points.
This article aims to provide an elementary and cross-disciplinary exposition for applications without any tensor calculus, and thus, we do not give any new mathematical results.
\end{abstract}

\section{Introduction}
In recent years, the application of differential geometry in architectural surface design has been actively studied, forming a large field called \textit{architectural geometry}.
This paper focuses on linear Weingarten surfaces, one of the important surface classes in architectural surface design, from the viewpoint of the variational principle.
The mechanical properties of linear Weingarten surfaces in the context of shell membrane theory can be found in \cite{Novozhilov,RogersSchief}.
The energy functional which has linear Weingarten surfaces as stationary points and related studies are classically well known, see for example \cite{Chen1,Chen2,Hsuing,Nitsche,PinlTrapp,Reilly} and the references therein.
Thus, this paper does not provide any new mathematical results.
However, it is essential to create a cross-disciplinary environment in which results in mathematics can be applied to architectural surface design.
Therefore, this paper is written to be accessible to structural engineering researchers with basic knowledge of differential geometry of surfaces, avoiding abstract Riemannian geometry or tensor calculus, and completing calculations as elementary as possible.
In the calculations, we show how the strain-displacement relation found in standard textbooks on classical thin shell theory can be used to derive the first variation formulas.

\section{Preliminaries}

In the theory of continuum mechanics, thin shells can be studied using classical differential geometry of surfaces by considering the middle surface of the shell as a smooth surface in the 3-dimensional Euclidean space.
For basic terminologies, see for example \cite{GreenZerna,Love,Novozhilov}.
In this paper, we use the symbols in \cite{Novozhilov,RogersSchief}.

Let us consider a surface $\bR = \bR (\alpha, \beta)$ in the 3-dimensional Euclidean space parametrized by the curvature line coordinates $(\alpha, \beta)$,
that is, the first and second fundamental forms of the surface are diagonalized.
We denote the first fundamental form $I$ by
\begin{equation}
I = A_1^2 \, d\alpha^2 + A_2^2 \, d\beta^2,\quad
A_1^2 = \bR_\alpha \cdot \bR_\alpha, \quad A_2^2 = \bR_\beta \cdot \bR_\beta,
\end{equation}
where the subscripts represent the partial derivatives.
Let $\be_1$ and $\be_2$ be the unit tangent vectors on the surface defined by the following relations:
\begin{equation}
\label{eq:tangent1}
\bR_\alpha = A_1 \be_1, \quad \bR_\beta = A_2 \be_2.
\end{equation}
Since the second fundamental form is also diagonalized, we have $\bR_{\alpha\beta} \perp \bN$, where $\bN$ is the unit normal vector on the surface.
A direct calculation shows that the integrability condition $(\bR_\alpha)_\beta = (\bR_\beta)_\alpha$ for the linear system \eqref{eq:tangent1} is given by
\begin{equation}
(A_1)_\beta = p A_2, \quad (A_2)_\alpha = q A_1, \quad 
(\be_1)_\beta = q \be_2, \quad (\be_2)_\alpha = p \be_1,
\end{equation}
where the first and second equations define $p$ and $q$.
Since $\bN$ is the unit vector, we have the relations
\begin{equation}
\bN_\alpha = H_\circ \be_1, \quad \bN_\beta = K_\circ \be_2,
\end{equation}
for some scalar values $H_\circ$ and $K_\circ$.
Then, the \textit{principal curvatures} $\kappa_1, \kappa_2$ and the \textit{principal curvature radii} $R_1, R_2$ are defined as follows:
\begin{equation}
H_\circ = - \kappa_1 A_1 = \frac{A_1}{R_1}, \quad K_\circ = - \kappa_2 A_2 = \frac{A_2}{R_2}.
\end{equation}
In particular, we have
\begin{equation}
\label{eq:Weingarten}
\bN_\alpha = - \kappa_1 \bR_\alpha, \quad
\bN_\beta = - \kappa_2 \bR_\beta.
\end{equation}
This formula is called the \textit{Weingarten formula} or \textit{Rodrigues formula}.
With these notations, the \textit{Gauss formula} is given by
\begin{equation}
(\be_1)_\alpha = - p \be_2 - H_\circ \bN, \quad (\be_2)_\alpha = p \be_1, \quad 
(\be_1)_\beta = q \be_2, \quad (\be_2)_\beta = - q \be_1 - K_\circ \bN,
\end{equation}
Or equivalently, we have the \textit{Gauss-Weingarten formula} in the matrix form:
\begin{equation}
\begin{aligned}
(\be_1, \be_2, \bN)_\alpha &= (\be_1, \be_2, \bN)
\begin{pmatrix}
0 & p & H_\circ \\
- p & 0 & 0 \\
- H_\circ & 0 & 0
\end{pmatrix}, \\
(\be_1, \be_2, \bN)_\beta &= (\be_1, \be_2, \bN)
\begin{pmatrix}
0 & - q & 0 \\
q & 0 & K_\circ \\
0 & - K_\circ & 0
\end{pmatrix}.
\end{aligned}
\end{equation}
The integrability condition for the Gauss-Weingarten system is called the \textit{Gauss-Mainardi-Codazzi equations}, which can be represented as follows:
\begin{equation}
p_\beta + q_\alpha + H_\circ K_\circ = 0, \quad
(H_\circ)_\beta = p K_\circ, \quad
(K_\circ)_\alpha = q H_\circ.
\end{equation}

\section{The strain-displacement relations}

If we consider the surface $\bR = \bR(\alpha, \beta)$ as the middle surface of a thin shell model, then we can consider the deformation of the middle surface as follows:
\begin{equation}
\br = \bR +  \Delta, \quad \Delta = u \be_1 + v \be_2 + w \bN,
\end{equation}
where $\Delta$ is called the \textit{displacement vector}.
For the deformed surface $\br$, we have the following relations:
\begin{equation}
\label{eq:deform_tangent}
\frac{1}{A_1} \br_\alpha = (1 + \varepsilon_1) \be_1 + \omega_1 \be_2 - \vartheta \bN, \quad
\frac{1}{A_2} \br_\beta = \omega_2 \be_1 + (1 + \varepsilon_2) \be_2 - \psi \bN,
\end{equation}
where
\begin{align}
\begin{aligned}
\varepsilon_1 &= \frac{1}{A_1} (u_\alpha + p v + H_\circ w), \quad
\varepsilon_2 = \frac{1}{A_2} (v_\beta + qu + K_\circ w), \\
\omega_1 &= \frac{1}{A_1} (v_\alpha - pu), \quad
\omega_2 = \frac{1}{A_2} (u_\beta - q v), \\
\vartheta &= \frac{1}{A_1} (- w_\alpha + H_\circ u), \quad
\psi = \frac{1}{A_2} (- w_\beta + K_\circ v).
\end{aligned}
\end{align}
The above relations are called the \textit{strain-displacement relations}.
Here, $\varepsilon_1$ and $\varepsilon_2$ are called the \textit{normal strains}, and $\omega = \omega_1 + \omega_2$ is called the \textit{shear strain}.

Next, we introduce the ``infinitesimal'' version of the strain-displacement relations.
We replace the displacement vector $\Delta$ with the variation vector $\bV$: 
\begin{equation}
\br = \bR +  t \cdot \bV, \quad \bV = v_1 \be_1 + v_2 \be_2 + v_n \bN, \quad |t| \ll 1,
\end{equation}
where $t$ is the deformation parameter.
For convenience, we decompose the variation vector $\bV$ to the tangential component $\etab$ and the normal component as follows:
\begin{equation}
\bV = \etab + v_n \bN, \quad \etab = v_1 \be_1 + v_2 \be_2.
\end{equation}
In the same way as before, we introduce the following symbols:
\begin{align}
\begin{aligned}
\overline{\varepsilon}_1 &= \frac{1}{A_1} ((v_1)_\alpha + p v_2 + H_\circ v_n), \quad
\overline{\varepsilon}_2 = \frac{1}{A_2} ((v_2)_\beta + q v_1 + K_\circ v_n), \\
\overline{\omega}_1 &= \frac{1}{A_1} ((v_2)_\alpha - p v_1), \quad
\overline{\omega}_2 = \frac{1}{A_2} ((v_1)_\beta - q v_2), \\
\overline{\vartheta} &= \frac{1}{A_1} (- (v_n)_\alpha + H_\circ v_1), \quad
\overline{\psi} = \frac{1}{A_2} (- (v_n)_\beta + K_\circ v_2).
\end{aligned}
\end{align}
In the following, we will refer to these relations as the \textit{infinitesimal strain-displacement relations}.

\section{The first variation of the area}

Let $M$ be a $2$-dimensional manifold with boundary $\partial M$ and $\bR: M \to \R^3$ be a smooth immersion.
Then the area of the immersion can be defined as follows:
\begin{equation}
A = \int_M dA,
\end{equation}
where the area element $dA$ is defined by $dA = \| \bR_\alpha \times \bR_\beta \| \, d\alpha d\beta$ in the local coordinates $(\alpha, \beta)$.
We denote the first fundamental form $I'$ of the deformed surface $\br$ as follows:
\begin{equation}
I' = A_1^{\prime2} \, d\alpha^2 + A_2^{\prime2} \, d\beta^2, \quad
A_1^{\prime2} = \br_\alpha \cdot \br_\alpha, \quad
A_2^{\prime2} = \br_\beta \cdot \br_\beta.
\end{equation}
Let $dA'$ be the area element of $\br$.
\begin{lem}
The first variation of the area element can be written as the summation of the infinitesimal normal strains:
\begin{equation}
\delta dA = \left. \frac{d}{dt} \right|_{t=0} dA = (\overline{\varepsilon}_1 + \overline{\varepsilon}_2) \, dA.
\end{equation}
\end{lem}
\begin{proof}
By using the relation \eqref{eq:deform_tangent}, we have
\begin{equation}
\begin{aligned}
\frac{1}{A_1A_2} (\br_\alpha \times \br_\beta)
= \vartheta \be_1 + \psi \be_2 + (1 + \varepsilon_1 + \varepsilon_2) \bN + O(t^2).
\end{aligned}
\end{equation}
Therefore we have
\begin{equation}
\begin{aligned}
dA'
&= \| \br_\alpha \times \br_\beta \| \, d\alpha d\beta\\
&= (\vartheta^2 + \psi^2 + (1 + \varepsilon_1 + \varepsilon_2)^2)^{1/2} \, A_1 A_2 \, d\alpha d\beta + O(t^2)\\
&= \left( 1 + \frac{1}{2} (2 (\varepsilon_1 + \varepsilon_2) + (\varepsilon_1 + \varepsilon_2)^2 + \vartheta^2 + \psi^2) \right) dA + O(t^2)\\
&= (1 + \varepsilon_1 + \varepsilon_2) \, dA + O(t^2).
\end{aligned}
\end{equation}
Here if we replace $(u, v, w)$ with $(tv_1, tv_2, tv_n)$, then we have
\begin{equation}
dA' = dA + t (\overline{\varepsilon}_1 + \overline{\varepsilon}_2) \, dA + O(t^2) \iff
\frac{dA' - dA}{t} = (\overline{\varepsilon}_1 + \overline{\varepsilon}_2) \, dA + O(t),
\end{equation}
which proves the statement.
\end{proof}
\begin{lem}
The summation of the infinitesimal normal strains is given by
\begin{equation}
\overline{\varepsilon}_1 + \overline{\varepsilon}_2 = \div \etab - 2 \cH v_n, \quad
\cH = \frac{1}{2} (\kappa_1 + \kappa_2),
\end{equation}
where $\cH$ is the mean curvature of the middle surface $\bR$.
\end{lem}
\begin{proof}
Note that the divergence of $\etab = v_1 \be_1 + v_2 \be_2$ is defined by
\begin{equation}
\div \etab = \frac{1}{A_1A_2} \left( (v_1 A_2)_\alpha + (v_2 A_1)_\beta \right).
\end{equation}
By using the relations $(A_1)_\beta = p A_2$, $(A_2)_\alpha = q A_1$, we have
\begin{equation}
\begin{aligned}
\div \etab
&= \frac{1}{A_1 A_2} ((v_1)_\alpha A_2 + v_1 q A_1 + (v_2)_\beta A_1 + v_2 p A_2)\\
&= \frac{1}{A_1} ((v_1)_\alpha + pv_2) + \frac{1}{A_2} ((v_2)_\beta + qv_1).
\end{aligned}
\end{equation}
Since $H_\circ = -\kappa_1 A_1$, $K_\circ = -\kappa_2 A_2$, we have
\begin{equation}
\begin{aligned}
\overline{\varepsilon}_1 + \overline{\varepsilon}_2
&= \frac{1}{A_1} ((v_1)_\alpha + pv_2 + H_\circ v_n) + \frac{1}{A_2} ((v_2)_\beta + qv_1 + K_\circ v_n) \\
&= \frac{1}{A_1} ((v_1)_\alpha + pv_2) + \frac{1}{A_2} ((v_2)_\beta + qv_1) - (\kappa_1 + \kappa_2) v_n\\
&= \div \etab - 2 \cH v_n.
\end{aligned}
\end{equation}
This proves the statement.
\end{proof}
\begin{thm}
\label{thm:area}
The first variation of the area is given by
\begin{equation}
\delta \int_M dA = - 2 \int_M \cH (\bV \cdot \bN) \, dA + \int_{\partial M} \etab \cdot \bt \, ds,
\end{equation}
where $\bt$ is the outward-pointing unit conormal vector and $ds$ is the line element of the boundary $\partial M$, respectively.
\end{thm}
\begin{proof}
The divergence theorem on the surface
\begin{equation}
\int_M \div \etab \, dA = \int_{\partial M} \etab \cdot \bt \, ds,
\end{equation}
immediately shows the statement.
\end{proof}

\section{The first variation of the volume}

The (algebraic) volume $V$ enclosed by the surface $\bR$ is defined by
\begin{equation}
V = \frac{1}{3} \int_M (\bR \cdot \bN) \, dA,
\end{equation}
where the integrand represents the volume of the ``infinitesimal cone''.
In the local coordinates $(\alpha, \beta)$, the integrand can be written as
\begin{equation}
(\bR \cdot \bN) \, dA = \bR \cdot (\bR_\alpha \times \bR_\beta) \, d\alpha d\beta,
\end{equation}
To derive the first variation formula of the volume, we will show two lemmas.
\begin{lem}
\label{lem:vol1}
The following relation holds
\begin{equation}
\div (\bR - (\bR \cdot \bN) \bN) = 2 (1 + \cH (\bR \cdot \bN)).
\end{equation}
\end{lem}
\begin{proof}
Using the Gauss formula, we have
\begin{equation}
\begin{aligned}
&\quad \div (\bR - (\bR \cdot \bN) \bN)\\
&= \frac{1}{A_1 A_2} ((\bR \cdot \be_1)_\alpha A_2 + (\bR \cdot \be_1) q A_1 + (\bR \cdot \be_2)_\beta A_1 + (\bR \cdot \be_2) pA_2)\\
&= \frac{1}{A_1 A_2} (A_1 A_2 + \bR \cdot (- p \be_2 - H_\circ \bN) A_2 + (\bR \cdot \be_1) q A_1\\
&\qquad+ A_2 A_1 + \bR \cdot (- q\be_1 - K_\circ \bN) A_1 + (\bR \cdot \be_2) pA_2)\\
&= \frac{1}{A_1A_2} (2 A_1 A_2 - (H_\circ A_2 + K_\circ A_1) (\bR \cdot \bN))\\
&= 2 + (\kappa_1 + \kappa_2) (\bR \cdot\bN) = 2 (1 + \cH (\bR \cdot \bN)).
\end{aligned}
\end{equation}
This proves the lemma.
\end{proof}
%
%
\begin{lem}
\label{lem:vol2}
In the local coordinates $(\alpha,\beta)$, the following relation holds
\begin{equation}
\bR \cdot \delta (\bR_\alpha \times \bR_\beta)\, d\alpha d\beta = (\div (- v_n \bR + (\bR \cdot \bN) \etab) + 2v_n) \, dA.
\end{equation}
\end{lem}
\begin{proof}
We first show the following formula:
\begin{equation}
\label{eq:step1}
\delta (\bR_\alpha \times \bR_\beta) = - A_1 A_2 \nabla v_n + (\div \etab - 2 \cH v_n) (\bR_\alpha \times \bR_\beta) + v_1 A_2 \bN_\alpha + v_2 A_1 \bN_\beta.
\end{equation}
Using the equation \eqref{eq:tangent1} and $\bR_\alpha \times \bR_\beta = A_1 A_2 \bN$, we have
\begin{equation}
\begin{aligned}
&\qquad \br_\alpha \times \br_\beta = t \overline{\vartheta} A_2 \bR_\alpha + t \overline{\psi} A_1 \bR_\beta 
+ \bR_\alpha \times \bR_\beta + t (\overline{\varepsilon}_1 + \overline{\varepsilon}_2) (\bR_\alpha \times \bR_\beta),\\
&\iff \frac{\br_\alpha \times \br_\beta - \bR_\alpha \times \bR_\beta}{t} 
= \overline{\vartheta} A_2 \bR_\alpha + \overline{\psi} A_1 \bR_\beta 
    + (\overline{\varepsilon}_1 + \overline{\varepsilon}_2) (\bR_\alpha \times \bR_\beta).
\end{aligned}
\end{equation}
Therefore we have
\begin{equation}
\delta (\bR_\alpha \times \bR_\beta) 
= \overline{\vartheta} A_2 \bR_\alpha + \overline{\psi} A_1 \bR_\beta 
    + (\overline{\varepsilon}_1 + \overline{\varepsilon}_2) (\bR_\alpha \times \bR_\beta).
\end{equation}
Since
\begin{equation}
\begin{aligned}
&\qquad \nabla v_n 
= \frac{(v_n)_\alpha}{A_1} \be_1 + \frac{(v_n)_\beta}{A_2} \be_2 
= (- \overline{\vartheta} - \kappa_1 v_1) \be_1 + (- \overline{\psi} - \kappa_2 v_2) \be_2,\\
&\iff A_1 A_2 \nabla v_n 
= - \overline{\vartheta} A_2 \bR_\alpha - \overline{\psi} A_1 \bR_\beta 
    + v_1 A_2 (- \kappa_1 \bR_\alpha) + v_2 A_1 (- \kappa_2 \bR_\beta),
\end{aligned}
\end{equation}
the Weingarten formula \eqref{eq:Weingarten} shows
\begin{equation}
\overline{\vartheta} A_2 \bR_\alpha + \overline{\psi} A_1 \bR_\beta 
= - A_1 A_2 \nabla v_n + v_1 A_2 \bN_\alpha + v_2 A_1 \bN_\beta.
\end{equation}
This proves the equation \eqref{eq:step1}.
Next, we show the equation:
\begin{equation}
\label{eq:step2}
\bR \cdot (v_1 A_2 \bN_\alpha + v_2 A_1 \bN_\beta) 
= A_1 A_2 \nabla (\bR \cdot \bN) \cdot \etab.
\end{equation}
A direct calculation shows
\begin{equation}
\begin{aligned}
&\quad \bR \cdot (v_1 A_2 \bN_\alpha + v_2 A_1 \bN_\beta)\\
&= (v_1 A_2 \bR \cdot \bN)_\alpha - (v_1 A_2 \bR)_\alpha \cdot \bN + (v_2 A_1 \bR)_\beta \cdot \bN \\
&= A_1 A_2 \div ((\bR \cdot \bN) \etab) - ((v_1 A_2)_\alpha + (v_2 A_1)_\beta) (\bR \cdot \bN) \\
&= A_1 A_2 (\div ((\bR \cdot \bN) \etab) - (\div \etab) (\bR \cdot \bN))
= A_1 A_2 (\nabla (\bR \cdot \bN) \cdot \etab),
\end{aligned}
\end{equation}
which proves the equation \eqref{eq:step2}.
Combining the equations \eqref{eq:step1} and \eqref{eq:step2}, we conclude
\begin{equation}
\begin{aligned}
&\quad \bR \cdot \delta (\bR_\alpha \times \bR_\beta) \\
&= - A_1 A_2 \bR \cdot \nabla v_n + (\div \etab - 2 \cH v_n) \bR \cdot (\bR_\alpha \times \bR_\beta)\\
&\qquad+A_1 A_2 \div ((\bR \cdot \bN)\etab) - (\div \etab) \bR \cdot (\bR_\alpha \times \bR_\beta) \\
&= - A_1 A_2 \bR \cdot \nabla v_n - (2 \cH (\bR \cdot \bN)) A_1 A_2 v_n + A_1 A_2 \div ((\bR \cdot \bN) \etab) \\
&= - A_1 A_2 \bR \cdot \nabla v_n - (\div (\bR - (\bR \cdot \bN) \bN) - 2) A_1 A_2 v_n + A_1 A_2 \div ((\bR \cdot \bN) \etab)\\
&= A_1 A_2 (\div (- v_n \bR + (\bR \cdot \bN) \etab) + 2 v_n).
\end{aligned}
\end{equation}
This proves the lemma.
\end{proof}
\begin{thm}
\label{thm:volume}
The first variation of the volume $V$ is given by
\begin{equation}
\delta V = \int_M v_n \, dA 
    + \frac{1}{3} \int_{\partial M} ((\bR \cdot \bN)\etab - v_n \bR) \cdot \bt \, ds.
\end{equation}
\end{thm}
\begin{proof}
It follows from Lemma \ref{lem:vol1} and \ref{lem:vol2} that we have
\begin{equation}
\begin{aligned}
\delta ((\bR \cdot \bN) \, dA)
&= \delta (\bR \cdot (\bR_\alpha \times \bR_\beta) \, d\alpha d\beta)\\
&= \bV \cdot (\bR_\alpha \times \bR_\beta) 
    + \bR \cdot \delta (\bR_\alpha \times \bR_\beta)) \, d\alpha d\beta\\
&= (3v_n + \div (- v_n \bR + (\bR \cdot \bN) \etab)) \, dA,
\end{aligned}
\end{equation}
in the local coordinates $(\alpha, \beta)$.
By using the divergence theorem, we have
\begin{equation}
\begin{aligned}
\delta V 
&= \frac{1}{3} \int_M (3v_n + \div (- (\bV \cdot \bN) \bR + v_n \etab)) \, dA\\
&= \int_M v_n \, dA + \frac{1}{3} \int_M \div (- v_n \bR + (\bR \cdot \bN) \etab) \, dA \\
&= \int_M v_n \, dA + \frac{1}{3} \int_{\partial M} (- v_n \bR + (\bR \cdot \bN) \etab) \cdot \bt \, ds
\end{aligned}
\end{equation}
which proves the statement.
\end{proof}

\section{The first variation of the mean curvature}

\begin{lem}
The first variations of the principal curvatures are given by
\begin{equation}
\delta \kappa_1 = - \frac{\overline{\vartheta}_\alpha + p \overline{\psi}}{A_1} - \overline{\varepsilon}_1 \kappa_1, \quad
\delta \kappa_2 = - \frac{\overline{\psi}_\beta + q \overline{\vartheta}}{A_2} - \overline{\varepsilon}_2 \kappa_2.
\end{equation}
\end{lem}
\begin{proof}
The frame $(\be_1', \be_2', \bN')$ on the deformed surface $\br$ is given by (see \cite{Novozhilov}, p.17):
\begin{equation}
\begin{aligned}
\be_1' &= \be_1 + \omega_1 \be_2 - \vartheta \bN + O(t^2), \\
\be_2' &= \omega_2 \be_1 + \be_2 - \psi \bN + O(t^2), \\
\bN' &= \vartheta \be_1 + \psi \be_2 + \bN + O(t^2).
\end{aligned}
\end{equation}
The principal curvatures $\kappa_1', \kappa_2'$ for the deformed surface $\br$ can be computed by
\begin{equation}
\kappa_1' = - \frac{\bN_\alpha' \cdot \be_1'}{A_1'}, \quad
\kappa_2' = - \frac{\bN_\beta' \cdot \be_2'}{A_2'}.
\end{equation}
By using the Gauss-Weingarten formula, we have
\begin{equation}
\begin{aligned}
\bN_\alpha' &= (\vartheta_\alpha + p\psi + H_\circ) \be_1 + (\psi_\alpha - p \vartheta) \be_2 - \vartheta H_\circ \bN +O(t^2),\\
\bN_\beta' &= (\vartheta_\beta - q\psi) \be_1 + (\psi_\beta + q \vartheta + K_\circ) \be_2 - \psi K_\circ \bN + O(t^2).
\end{aligned}
\end{equation}
Then we have
\begin{equation}
\bN_\alpha' \cdot \be_1' = \vartheta_\alpha + p \psi + H_\circ + O(t^2), \quad
\bN_\beta' \cdot \be_2' = \psi_\beta + q \vartheta + K_\circ + O(t^2).
\end{equation}
Therefore
\begin{equation}
\begin{aligned}
\kappa_1' 
= - \frac{\vartheta_\alpha + p \psi + H_\circ}{A_1 (1 + \varepsilon_1)} + O(t^2)
&= - \frac{\vartheta_\alpha + p \psi + H_\circ}{A_1} (1 - \varepsilon_1) + O(t^2)\\
&= - \frac{\vartheta_\alpha + p \psi + H_\circ}{A_1} + \varepsilon_1 \frac{H_\circ}{A_1} + O(t^2), \\
\kappa_2'
= - \frac{\psi_\beta + q \vartheta + K_\circ}{A_2 (1 + \varepsilon_2)} + O(t^2)
&= - \frac{\psi_\beta + q \vartheta + K_\circ}{A_2} (1 - \varepsilon_2) + O(t^2)\\
&= - \frac{\psi_\beta + q \vartheta + K_\circ}{A_2} + \varepsilon_2 \frac{K_\circ}{A_2} + O(t^2).
\end{aligned}
\end{equation}
Or equivalently,
\begin{equation}
\label{eq:diff_curvature}
\kappa_1' - \kappa_1 = - \frac{\vartheta_\alpha + p \psi}{A_1} - \varepsilon_1 \kappa_1 + O(t^2), \quad
\kappa_2' - \kappa_2 = - \frac{\psi_\beta + q \vartheta}{A_2} - \varepsilon_2 \kappa_2 + O(t^2).
\end{equation}
The statement follows from the replacement $(u, v, w) \to (tv_1, tv_2, tv_n)$.
\end{proof}
\begin{rmk}
The equation \eqref{eq:diff_curvature} is essentially pointed out in \cite{Novozhilov}, p.25.
\end{rmk}
\begin{lem}
\label{lem:mean1}
The first variation of the mean curvature is given by
\begin{equation}
\delta \cH = \frac{1}{2} \div \nabla v_n + \nabla \cH \cdot \etab + (2 \cH^2 - \cK) v_n,
\end{equation}
where $\cK = \kappa_1 \kappa_2$ is the Gaussian curvature.
\end{lem}
\begin{proof}
If we note the relation
\begin{equation}
\nabla v_n = (- \overline{\vartheta} - \kappa_1 v_1) \be_1 + (- \overline{\psi} - \kappa_2 v_2) \be_2,
\end{equation}
then we have
\begin{equation}
\begin{aligned}
&\quad \delta (\kappa_1 + \kappa_2)\\
&= - \frac{\overline{\vartheta}_\alpha + p \overline{\psi}}{A_1} 
    - \frac{\overline{\psi}_\beta + q \overline{\vartheta}}{A_2} 
        - \overline{\varepsilon}_1 \kappa_2 - \overline{\varepsilon}_2 \kappa_1\\
&= - \div (\overline{\vartheta} \be_1 + \overline{\psi} \be_2) 
    - (\overline{\varepsilon}_1 + \overline{\varepsilon}_2)(\kappa_1 + \kappa_2) 
        + \overline{\varepsilon}_1 \kappa_2 + \overline{\varepsilon}_2 \kappa_1\\
&= \div (\nabla v_n + \kappa_1 v_1 \be_1 + \kappa_2 v_2 \be_2) 
    - 2 \cH (\div \etab - 2 \cH v_n) 
        + \overline{\varepsilon}_1 \kappa_2 + \overline{\varepsilon}_2 \kappa_1\\
&= \div (\nabla v_n + 2 \cH \etab - \kappa_2 v_1 \be_1 - \kappa_1 v_2 \be_2) 
    - 2\cH \div \etab + 4 \cH^2 v_n 
        + \overline{\varepsilon}_1 \kappa_2 + \overline{\varepsilon}_2 \kappa_1\\
&= \div \nabla v_n + 2 \nabla \cH \cdot \etab + 4 \cH^2 v_n 
    - \div (\kappa_2 v_1 \be_1 + \kappa_1 v_2 \be_2) 
        + \overline{\varepsilon}_1 \kappa_2 + \overline{\varepsilon}_2 \kappa_1.
\end{aligned}
\end{equation}
If we show the equation
\begin{equation}
\label{eq:key1}
\div (\kappa_2 v_1 \be_1 + \kappa_1 v_2 \be_2) - \overline{\varepsilon}_1 \kappa_2 - \overline{\varepsilon}_2 \kappa_1 = 2 \cK v_n,
\end{equation}
then
\begin{equation}
\delta (\kappa_1 + \kappa_2) = \div \nabla v_n + 2 \nabla \cH \cdot \etab + (4 \cH^2 - 2 \cK) v_n,
\end{equation}
which proves the statement.
Applying the Mainardi-Codazzi equations, we have
{\normalsize
\begin{equation}
\begin{aligned}
&\qquad \div (\kappa_2 v_1 \be_1 + \kappa_1 v_2 \be_2) 
    - \overline{\varepsilon}_1 \kappa_2 - \overline{\varepsilon}_2 \kappa_1\\
&= \frac{1}{A_1} ((\kappa_2 v_1)_\alpha + p \kappa_1 v_2) 
    + \frac{1}{A_2} ((\kappa_1 v_2)_\beta + q \kappa_2 v_1)\\
&\qquad - \frac{\kappa_2}{A_1} ((v_1)_\alpha + pv_2 - \kappa_1 A_1 v_n) 
    - \frac{\kappa_1}{A_2} ((v_2)_\beta + q v_1 - \kappa_2 A_2 v_n)\\
&= \frac{1}{A_1A_2} ((\kappa_2)_\alpha A_2 v_1 + p (K_\circ + \kappa_1 A_2) v_2 
    + (\kappa_1)_\beta A_1 v_2 + q (H_\circ + \kappa_2 A_1)v_1) + 2 \cK v_n\\
&= \frac{1}{A_1A_2} ((q H_\circ - (K_\circ)_\alpha)v_1 + (p K_\circ - (H_\circ)_\beta) v_2) + 2 \cK v_n
= 2 \cK v_n.
\end{aligned}
\end{equation}
}
This shows the equation \eqref{eq:key1} and proves the lemma.
\end{proof}
\begin{thm}
\label{thm:mean}
The first variation of the integral of the mean curvature is given by
\begin{equation}
\delta \int_M \cH \, dA = - \int_M \cK v_n \, dA + \int_{\partial M} \left( \frac{1}{2} \nabla v_n + \cH \etab \right) \cdot \bt \, ds.
\end{equation}
\end{thm}
\begin{proof}
It follows from Lemma \ref{lem:mean1} and the divergence theorem that we have
\begin{equation}
\begin{aligned}
&\qquad \delta \int_M \cH \, dA\\
&= \int_M (\delta \cH) \, dA + \cH \, \delta dA \\
&= \int_M \left( \frac{1}{2} \div \nabla v_n + \nabla \cH \cdot \etab + (2 \cH^2 - \cK) v_n + \cH \div \etab - 2 \cH^2 v_n \right) \, dA\\
&= \int_M \left( \div \left( \frac{1}{2} \nabla v_n + \cH \etab \right) - \cK v_n \right) \, dA\\
&= - \int_M \cK v_n \, dA + \int \left( \frac{1}{2} \nabla v_n + \cH \etab \right) \cdot \bt \, ds.
\end{aligned}
\end{equation}
\end{proof}
Combining the Theorem \ref{thm:area}, \ref{thm:volume} and \ref{thm:mean}, we have the following result:
\begin{thm}
For constants $a, b, c \in \R$, we define the functional $E$ as follows:
\begin{equation}
E = \int_M (a \cH + b) \, dA + c V.
\end{equation}
Then the first variation formula is given by
\begin{equation}
\begin{aligned}
\delta E &= - \int_M (a \cK + 2 b \cH - c) v_n \, dA\\
&\quad+ \int_{\partial M} \left[ a \left( \frac{1}{2} \nabla v_n + \cH \etab \right) + b \etab + \frac{c}{3} ((\bR \cdot \bN)\etab - v_n \bR) \right] \cdot \bt \,ds.
\end{aligned}
\end{equation}
In particular, the equilibrium condition away from the boundary is given by
\begin{equation}
a \cK + 2 b \cH = c,
\end{equation}
that is, we have the linear Weingarten condition.
\end{thm}

\section*{Acknowledgments}
This study is supported by JST CREST Grant No. JPMJCR1911.
The author would like to thank Prof. Makoto Ohsaki at Kyoto University for his valuable comments.

\bibliographystyle{amsplain}
\bibliography{main}
\end{document}